\documentclass[11pt]{amsart}

\usepackage{amsmath,amssymb,amscd,amsthm,amsxtra}
\usepackage[dvips]{graphics,epsfig}

\allowdisplaybreaks[2]

\newtheorem{theorem}{Theorem} [section]
\newtheorem{maintheorem}{Theorem}
\newtheorem{lemma}[theorem]{Lemma}
\newtheorem{proposition}[theorem]{Proposition}
\newtheorem{remark}[theorem]{Remark} %{Remark}[section]

\DeclareMathOperator*{\intt}{\int}

\DeclareMathOperator{\MAX}{MAX}

\newcommand{\noi}{\noindent}
\newcommand{\Z}{\mathbb{Z}}
\newcommand{\R}{\mathbb{R}}

\newcommand{\T}{\mathbb{T}}

\newcommand{\al}{\alpha}
\newcommand{\dl}{\delta}
\newcommand{\Dl}{\Delta}
\newcommand{\eps}{\varepsilon}

\newcommand{\g}{\gamma}

\newcommand{\ld}{\lambda}

\newcommand{\s}{\sigma}
\newcommand{\ft}{\widehat}
\newcommand{\wt}{\widetilde}
\newcommand{\cj}{\overline}
\newcommand{\dx}{\partial_x}

\newcommand{\LRA}{\Longrightarrow}

\newcommand{\jb}[1]
{\langle #1 \rangle}

\begin{document}

\title
[ Invariance of  white noise for the periodic KdV]
{\bf Invariance of the white noise for KdV}

\author{Tadahiro Oh}

\address{Tadahiro Oh\\
Department of Mathematics\\
University of Toronto\\
40 St. George St, Rm 6290,
Toronto, ON M5S 2E4, Canada}

\email{oh@math.toronto.edu}

%\date{\today}



\begin{abstract}
We prove the invariance of the mean 0 white noise for the periodic KdV.
First, we show that the Besov-type space $\ft{b}^s_{p, \infty}$, $sp <-1$, 
contains the support of the white noise.
Then, we prove  local well-posedness in 
$\ft{b}^s_{p, \infty}$ for $p= 2+$, $s = -\frac{1}{2}+$ such that $sp <-1$.
In establishing the local well-posedness, 
we use a variant of the Bourgain spaces with a weight.
This provides an analytical proof of the invariance of the white noise under the flow of KdV obtained in 
Quastel-Valko \cite{QV}.
\end{abstract}

\maketitle

%\tableofcontents

\section{Introduction}

In this paper, we  consider the periodic Korteweg-de Vries (KdV) equation:
\begin{equation} \label{KDV}
\begin{cases}
u_t + u_{xxx} +  u u_x  = 0 \\ 
u \big|_{t = 0} = u_0,
\end{cases}
\end{equation}

\noi 
where $u$ is a real-valued function on $\T\times\R$ with $\T = [0, 2\pi)$ and 
the mean of $u_0$ is 0.
By the conservation of the mean, it follows that the solution $u(t)$ of \eqref{KDV} has the spatial mean 0
for all $t \in\mathbb{R}$ as long as it exists.
Our main goal is to show that the mean 0 white noise  
\begin{equation}\label{white}
d \mu = Z^{-1}\exp(- \tfrac{1}{2} \int u^2 dx) \prod_{x\in \T} d u(x), \ u \text{ mean } 0  
\end{equation}

\noi
is invariant under the flow
and that \eqref{KDV} is globally well-posed almost surely on the statistical ensemble (i.e. on the support of $\mu$)
without using the complete integrability of the equation.

\medskip

First, we briefly review recent well-posedness results of the periodic KdV \eqref{KDV}.
In \cite{BO1}, Bourgain  introduced a new weighted space-time Sobolev space $X^{s, b}$
whose norm is given by
\begin{equation} \label{Xsb}
\| u \|_{X^{s, b}(\mathbb{T} \times \mathbb{R})} = \big\| \jb{n}^s \jb{\tau - n^3}^b 
\ft{u}(n, \tau) \big\|_{L^2_{n, \tau}(\mathbb{Z} \times \R)},
\end{equation}

\noindent
where $\jb{ \: \cdot \:} = 1 + |  \cdot  | $. 
He proved the local well-posedness of \eqref{KDV} in $L^2(\mathbb{T})$
via the fixed point argument, 
immediately yielding the global well-posedness in $L^2(\mathbb{T})$
thanks to the conservation of the $L^2$ norm.
Kenig-Ponce-Vega \cite{KPV4} improved Bourgain's result 
and established the local well-posedness in $H^{-\frac{1}{2}}(\T)$.
Colliander-Keel-Staffilani-Takaoka-Tao \cite{CKSTT4} proved 
the corresponding global well-posedness result via the $I$-method. 
More recently, Kappeler-Topalov \cite{KT} proved the global well-posedness of the KdV in $H^{-1}(\T)$, 
using the complete integrability of the equation. 

There are also results on the necessary conditions on the  regularity 
with respect to smoothness or uniform continuity 
of the solution map $: u_0 \in H^s (\mathbb{T}) \to u(t) \in H^s(\mathbb{T})$.
Bourgain \cite{BO3} showed that if the solution map is $C^3$, 
then $s \geq -\frac{1}{2}$.  
Christ-Colliander-Tao \cite{CCT}
proved that if the solution map is uniformly continuous, 
then $s \geq -\frac{1}{2}$.
(Also, see Kenig-Ponce-Vega \cite{KPV5}.) 

\medskip

In \cite{BO4}, Bourgain proved the invariance of the Gibbs measures for the nonlinear
Schr\"odinger equations (NLS).  
In dealing with the super-cubic nonlinearity,
(where only the local well-posedness result was available),
he used a probabilistic argument and the approximating finite dimensional ODEs
(with the invariant finite dimensional Gibbs measures)
to extend the local solutions to the global ones almost surely on the statistical ensemble
and showed the invariance of the Gibbs measures.
Note that it was crucial that the local well-posedness was obtained with a ``good" estimate
on the solutions (e.g. via the fixed point argument)
for his argument to obtain the uniform convergence of the solutions 
of the finite dimensional ODEs to those of the full PDE.
Also see Burq-Tzvetkov \cite{BT1}, Oh \cite{OH3}, and Tzvetkov \cite{TZ1}, \cite{TZ2}.

In the present paper, we'd like to follow Bourgain's argument \cite{BO4}.
Unfortunately, it is known (c.f. Zhidkov \cite{Z}) that the white noise $\mu$ in \eqref{white}
is supported on $\cap_{s < -\frac{1}{2}} H^s \setminus H^{-\frac{1}{2}}$.
In view of the results in \cite{BO3} and \cite{CCT} described above, 
we can not hope to have a local-in-time solution via the fixed point argument in $H^s$, $s < -\frac{1}{2}$.
Instead, we will prove a local well-posedness in an appropriate Banach space containing the support of the white noise $\mu$.
Define the Besov-type space 
 via the norm
\begin{equation} \label{Besov}
\| f\|_{\ft{b}^s_{p, \infty}} 
:= \| \ft{f}\|_{b^s_{p, \infty}} = \sup_j \| \jb{n}^s \ft{f}(n) \|_{L^p_{|n|\sim 2^j}}
= \sup_j \Big( \sum_{|n| \sim 2^j} \jb{n}^{sp} |\ft{f}(n)|^p \Big)^\frac{1}{p}.
\end{equation}

\noi
By Hausdorff-Young's inequality, 
we have $\ft{b}^s_{p, \infty} \supset B^s_{p', \infty}$ for $p >2$,
where $B^s_{p', \infty}$ is the usual Besov space with $p' = \frac{p}{p-1}$.
In Section 3,  we use the theory of abstract Wiener spaces to show that 
$\ft{b}^s_{p, \infty}$ contains the full support of the white noise for $sp < -1$.

\medskip 

Now, we'd like to establish the local well-posedness in $\ft{b}^s_{p, \infty}$ for $sp < -1$.
Note that this space is essentially less regular than $H^{-\frac{1}{2}}$ since it contains the support of the white noise.
First, define a variant of the $X^{s, b}$ space adjusted to $\ft{b}^s_{p, \infty} (\mathbb{T})$.
Let $X^{s, b}_p$ be the completion of the Schwartz class $\mathcal{S}(\mathbb{T} \times \mathbb{R})$ under the norm
\begin{equation} \label{XSBP}
 \| u \|_{X^{s, b}_p} 
 =  \|\jb{n}^s \jb{\tau - n^3}^b \ft{u}(n, \tau)\|_{b^s_{p, \infty} L^p_\tau}.
\end{equation}

\noi
Then, one of the crucial bilinear estimates that we need to prove is:
\begin{equation} \label{WWbilinear1}
\| \dx(uv) \|_{X^{s, -\frac{1}{2}}_p} \lesssim \| u \|_{X^{s, \frac{1}{2}}_p} \| v \|_{X^{s, \frac{1}{2}}_p}. 
\end{equation}

\noi
As in \cite{BO1} and \cite{KPV4}, a key ingredient is the algebraic identity
$n^3 - n_1^3 - n_2^3 = 3 n n_1 n_2$ for $n = n_1 + n_2$.
However, this is not enough to prove \eqref{WWbilinear1} for $sp < -1$.
In establishing the local well-posedness through the usual integral equation, 
we view the nonlinear problem \eqref{KDV} 
as a perturbation to the Airy equation $u_t + u_{xxx} = 0$.
Noting the Fourier transform of the solution to the Airy equation is
 a measure supported on $\{ \tau =n^3\}$,
we modify $X^{s, b}_p$ with a carefully chosen weight $w(n, \tau)$ in Section 4
to treat the resonant cases in \eqref{WWbilinear1}. 
(c.f. Bejenaru-Tao \cite{BTAO}, Kishimoto \cite{KISHI} in the context of NLS.)

\begin{maintheorem} \label{THM:LWP2}
Assume the mean 0 condition on $u_0$.
Let $s = -\frac{1}{2}+$, $p = 2+$ such that $s p < -1$.
Then, KdV \eqref{KDV} is locally well-posed in $\ft{b}^s_{p, \infty}$.
\end{maintheorem}

Once we prove Theorem \ref{THM:LWP2}, 
we can use the finite dimensional approximation to \eqref{KDV}: 
\begin{equation} \label{KDVN}
\begin{cases}
u^N_t + u^N_{xxx} + \mathbb{P}_N (u^N u^N_x)  = 0 \\ 
u^N \big|_{t = 0} = u^N_0,
\end{cases}
\end{equation}

\noi
where $\mathbb{P_N}$ is the  projection onto the frequencies $|n| \leq N $
and $u^N = \mathbb{P}_N u$.
Note that \eqref{KDVN} is Hamiltonian, 
and that it preserves $\int (u^N)^2 dx$.
Hence, by Liouville's theorem, the finite dimensional white noise 
\begin{equation}\label{whiteN}
d \mu_N = Z_N^{-1}\exp(- \tfrac{1}{2} \int (u^N)^2 dx) \prod_{x\in \T} d u^N(x) 
\end{equation}
is invariant under the flow of \eqref{KDVN}.
The remaining argument follows just as in 
\cite{BO4}, \cite{BT1}, \cite{OH3}, \cite{TZ1}, \cite{TZ2},
and we obtain the a.s. GWP of \eqref{KDV} and the invariance of the white noise $\mu$.

\begin{maintheorem} \label{THM:GWP2}
Let $\{ g_n (\omega) \}_{n = 1}^\infty$ be a sequence of i.i.d. standard  complex Gaussian
random variables on a probability space $(\Omega, \mathcal{F}, P)$. 
Consider \eqref{KDV} with initial data 
$ u_0 = \sum_{n \ne 0} g_n (\omega) e^{inx},$
where $g_{-n} = \cj{g_n}$.
Then, \eqref{KDV} is globally well-posed almost surely in $\omega \in \Omega.$
Moreover, the mean 0 white noise $\mu$ is invariant under the flow. 
\end{maintheorem}

\begin{remark} \label{REM:KdVwhite} \rm
This provides an analytical proof of the invariance of the white noise $\mu$.
Recently, Quastel-Valko \cite{QV} proved  the invariance of the white noise under the flow of KdV.
Their argument  combines the  GWP in $H^{-1}(\T)$ via the complete integrability (Kappeler-Topalov \cite{KT}), 
the correspondence between the white noise for KdV and the Gibbs measure (weighted Wiener measure) of mKdV 
under the (corrected) Miura transform (Cambronero-McKean \cite{CM}),
and the invariance of the Gibbs measure of mKdV (Bourgain \cite{BO4}.)
Their method is not applicable to the general non-integrable coupled KdV system
considered in \cite{OH3}, whereas our argument is applicable in the non-integrable case as well.
\end{remark}

\begin{remark}
Let $\mathcal{F} L^{s, p}$ be the space of functions on $\mathbb{T}$ defined via the norm
$\|f\|_{\mathcal{F} L^{s, p}}=\| \jb{n}^s \ft{f}(n) \|_{L^p_n}$.
Then, Theorems \ref{THM:LWP2} and \ref{THM:GWP2} can also be established in 
$\mathcal{F} L^{s, p}$
for some $s = -\frac{1}{2}+$, $p = 2+$ with $s p < -1$.
See Remark \ref{REM:FLP}.
\end{remark}

This paper is organized as follows:
In Section 2, we introduce some standard notations.
In Section 3, we go over the basic theory of Gaussian Hilbert spaces and abstract Wiener spaces.
Then, we give the precise mathematical meaning to the white noise $\mu$
and show that it is a (countably additive) probability measure on $\ft{b}^s_{p, \infty}$ for $sp < -1$.
In Section 4, we introduce the function spaces and linear estimates.
Then, we prove Theorem \ref{THM:LWP2} by establishing the crucial bilinear estimate.

\section{Notation}

In the periodic setting on $\T$, the spatial Fourier domain is $\Z$.
Let $dn$ be the normalized counting measure on $\Z$, 
and we say $f \in L^p(\Z)$, $1 \leq p < \infty$, if
\[ \| f \|_{L^p(\mathbb{Z})} = \bigg( \int_{\mathbb{Z}} |f(n)|^p dn \bigg)^\frac{1}{p}  
:= \bigg( \frac{1}{2\pi} \sum_{n \in \mathbb{Z}} |f(n)|^p \bigg)^\frac{1}{p} < \infty.\]

\noindent
If $ p = \infty$, we have the obvious definition involving the essential supremum.
We often drop $2\pi$ for simplicity.
If a function depends on both $x$ and $t$, we use ${}^{\wedge_x}$ 
(and ${}^{\wedge_t}$) to denote the spatial (and temporal) Fourier transform, respectively.
However, when there is no confusion, we simply use ${}^\wedge$ to denote the spatial Fourier transform,
the temporal Fourier transform, and  the space-time Fourier transform, depending on the context.

Given a space $X$ of functions on $\mathbb{T} \times \mathbb{R}$, 
we define the local in time restriction $X(\T \times I )$ for any time interval $I = [t_1, t_2]\subset \mathbb{R}$,
(or simply $X_{[t_1, t_2]}$) by
\[ \|u\|_{X_I} = \|u \|_{X(\T \times I )} = \inf \big\{ \|\wt{u} \|_{X(\T \times \mathbb{R})}: {\wt{u}|_I = u}\big\}.\]

\noindent
For a Banach space $X \subset \mathcal{S}'(\mathbb{T} \times \mathbb{R})$, 
we use $\ft{X}$ to denote the space of the Fourier transforms of the functions in $X$,
which is a Banach space with the norm $\|f\|_{\ft{X}} = \| \mathcal{F}^{-1}_{n, \tau} f\|_{X}$,
where $\mathcal{F}^{-1}$ denotes the inverse Fourier transform (in $n$ and $\tau$.)
Also, for a space $Y$ of functions on $\mathbb{Z}$, 
we use $\ft{Y}$ to denote the space of the inverse Fourier transforms of the functions in $Y$
 with the norm $\|f\|_{\ft{Y}} = \| \mathcal{F} f\|_{Y}$.

Now, define $\ft{b}^s_{p, q}(\mathbb{T})$ by the norm 
\begin{align} \label{Besov2}
\| f\|_{\ft{b}^s_{p, q}(\mathbb{T})} 
= \| \ft{f}\|_{b^s_{p, q}(\mathbb{Z})} 
: = \big\| \| \jb{n}^s \ft{f}(n) \|_{L^p_{|n|\sim 2^j}} \big\|_{l^q_j} 
= \Big( \sum_{j = 0}^\infty\Big( \sum_{|n| \sim 2^j} \jb{n}^{sp} |\ft{f}(n)|^p 
 \Big)^\frac{q}{p} \Big)^\frac{1}{q}
\end{align}

\noi 
for $ q < \infty$ and by \eqref{Besov} when $q = \infty$.

Lastly,
let $\eta \in C^\infty_c(\mathbb{R})$ be a smooth cutoff function supported on $[-1, 1]$ with $\eta \equiv 1$ 
on $[-\frac{1}{2}, \frac{1}{2}]$
and let $\eta_{_T}(t) =\eta(T^{-1}t)$. 
We use $c,$ $ C$ to denote various constants, usually depending only on $s$ and $p $.
If a constant depends on other quantities, we will make it explicit.
We use $A\lesssim B$ to denote an estimate of the form $A\leq CB$.
Similarly, we use $A\sim B$ to denote $A\lesssim B$ and $B\lesssim A$
and use $A\ll B$ when there is no general constant $C$ such that $B \leq CA$.
We also use $a+$ (and $a-$) to denote $a + \eps$ (and $a - \eps$), respectively,  
for arbitrarily small $\eps \ll 1$.

\section{Gaussian Measures in Hilbert Spaces and Abstract Wiener Spaces}

In this section, we go over the basic theory of Gaussian measures in Hilbert spaces
and abstract Wiener spaces
to provide the precise meaning of 
the white noise 
``$d \mu = Z^{-1}\exp(-\frac{1}{2} \int u^2 dx) \prod_{x\in \T} d u(x)$"
appearing in Section 1.
For details, see  Zhidokov \cite{Z}, Gross \cite{GROSS}, and Kuo \cite{KUO}.

First, recall  (centered) Gaussian measures in $\mathbb{R}^n$.
Let $n \in \mathbb{N}$ and $B$ be a symmetric positive $n \times n$ matrix
with real entries.
The  Borel measure $\mu$ in $\mathbb{R}^n$ with the density
\[ d \mu(x) = \frac{1}{\sqrt{(2\pi)^n \det (B )}} \exp \big( -\tfrac{1}{2} \langle B^{-1} x, x \rangle_{\mathbb{R}^n} \big)\]

\noindent
is called a (nondegenerate centered) Gaussian measure in $\mathbb{R}^n$.
Note that $\mu(\mathbb{R}^n) = 1$.

Now, we consider the analogous definition of the infinite dimensional (centered) Gaussian measures.
Let $H$ be a real separable Hilbert space and $B: H \to H$ be a linear positive self-adjoint operator 
(generally not bounded) with eigenvalues $\{\ld_n\}_{n\in \mathbb{N}}$
and the corresponding eigenvectors $\{e_n\}_{n\in\mathbb{N}}$  forming an orthonormal basis of $H$.
We call a set $M \subset H$  cylindrical if there exists an integer $n\geq 1$ and a Borel set $F \subset \mathbb{R}^n$
such that
\begin{equation} \label{CYLINDER}
 M = \big\{ x \in H : ( \jb{ x, e_1}_H, \cdots, \jb{ x, e_n}_H ) \in F \big\}. 
\end{equation}

\noindent
For a fixed operator $B$ as above, we denote by $\mathcal{A}$ the set of all cylindrical subsets of $H$.
One can easily verify that $\mathcal{A}$ is a field.
Then, the centered Gaussian measure in $H$ with the correlation operator $B$ is defined as 
the additive (but not countably additive in general) measure $\mu$ defined on the field $\mathcal{A}$
via
\begin{equation} \label{CGAUSSIAN}
 \mu(M) = (2\pi)^{-\frac{n}{2}} \prod_{j = 1}^n \ld_j^{-\frac{1}{2}} \int_F e^{-\frac{1}{2}\sum_{j = 1}^n \ld_j^{-1} x_j^2 }d x_1 \cdots dx_n,
\text{ for }M \in \mathcal{A} \text{ as in \eqref{CYLINDER}. }
\end{equation}

The following theorem tells us when  this Gaussian measure $\mu$ is countably additive.

\begin{theorem} \label{COUNTABLEADD}
The  Gaussian measure $\mu$ defined in \eqref{CGAUSSIAN} is countably additive
on the field $\mathcal{A}$ if and only if $B$ is an operator of trace class, 
i.e. $\sum_{n = 1}^\infty \ld_n < \infty$.
If the latter holds, then
the minimal $\s$-field $\mathcal{M}$ containing the field $\mathcal{A}$ of all cylindrical sets is the Borel $\s$-field on $H$.
\end{theorem}

Consider a sequence of the finite dimensional Gaussian measures $\{\mu_n\}_{n\in\mathbb{N}}$ as follows.
For fixed $n \in \mathbb{N}$, let $\mathcal{M}_n$ be the set of all cylindrical sets in $H$ of the form \eqref{CYLINDER} with this fixed $n$
and arbitrary Borel sets $F\subset \mathbb{R}^n$.
Clearly, $\mathcal{M}_n$ is a $\s$-field, and setting 
\[ \mu_n(M) = (2\pi)^{-\frac{n}{2}} \prod_{j = 1}^n \ld_j^{-\frac{1}{2}} \int_F e^{-\frac{1}{2}\sum_{j = 1}^n \ld_j^{-1} x_j^2 }d x_1 \cdots dx_n\]

\noindent
for $M \in \mathcal{M}_n$, we obtain a countably additive measure $\mu_n$ defined on $\mathcal{M}_n$.
Then, one can show that each measure $\mu_n$ can be naturally extended onto the whole Borel $\s$-field $\mathcal{M}$ of $H$
by $ \mu_n(A) := \mu_n(A \cap \text{span}\{e_1, \cdots, e_n\})$
for $A \in \mathcal{M}$.
Then, we have

\begin{proposition} \label{PROP:Zhidkov2}
Let $\mu$ in  \eqref{CGAUSSIAN} be countably additive.
Then,  $\{\mu_n\}_{n\in \mathbb{N}}$ constructed above converges weakly to $\mu$ as $n \to \infty$.

\end{proposition}

Now,  we construct the mean 0 white noise.
Let $\phi = \sum_{n} a_n e^{inx}$ be a real-valued function on $\mathbb{T}$ with mean 0. 
i.e. we have $a_0 = 0$ and $a_{-n} = \cj{a_n}$.
First, define $\mu_N$ on $\mathbb{C}^{N} \cong \mathbb{R}^{2N}$
with the  density
\begin{equation}\label{WhiteN}
 d \mu_N = Z_N^{-1} e^{-\frac{1}{2} \sum_{n = 1}^N  |a_n|^2 } \textstyle \prod_{ n = 1 }^N d a_n , 
\end{equation}

\noindent
where
$Z_{N} = \int_{\mathbb{C}^{N}} 
e^{-\frac{1}{2} \sum_{ n = 1 }^N  |a_n|^2} \prod_{ n = 1 }^N  d a_n . $
Note that this measure is the induced probability measure on $\mathbb{C}^{N}$ under the map
$ \omega \mapsto \{  g_n(\omega)  \}_{n = 1}^N,$
\noindent
where $g_n(\omega)$, $n = 1, \cdots, N$,  are i.i.d. standard complex Gaussian random variables.
Next, define the white noise $\mu$ by
\begin{equation} \label{White}
 d \mu = Z^{-1} e^{-\frac{1}{2} \sum_{  n \geq 1 }  |a_n|^2} \textstyle \prod_{  n \geq 1 } d  a_n , 
\end{equation}

\noindent
where
$Z = \int e^{-\frac{1}{2} \sum_{   n \geq 1 }  |a_n|^2}  \prod_{  n \geq 1 } d a_n. $
Then, in the above correspondence, we have 
$\phi = \sum_{n \ne 0} g_n e^{inx}$, 
where $\{g_n(\omega)\}_{n \geq 1}$ are i.i.d. standard complex Gaussian random variables
and $g_{-n} = \cj{g_n}$.

Let $\dot{H}^s_0$ be the homogeneous Sobolev space restricted to the {\it real-valued} mean 0 elements.
Let $\jb{\cdot, \cdot}_{\dot{H}^s_0}$ denote the inner product in $\dot{H}^s_0$.
i.e. $ \big\langle \sum c_n e^{inx}, \sum d_n e^{inx} \big\rangle_{\dot{H}_0^s} 
= \sum_{n \geq 1} |n|^{2s} c_n \cj{d_n} $.
Let $B_s = \sqrt{-\Dl}\vphantom{|}^{2s}$.
Then, the weighted exponentials $\{|n|^{-s} e^{inx}\}_{n\ne 0}$ are the eigenvectors of $B_s$ with the eigenvalue $|n|^{2s}$,
forming an orthonormal basis of $\dot{H}^s_0$.
Note that
\[ -\tfrac{1}{2}\jb{B^{-1} \phi, \phi}_{\dot{H_0^s}} 
= -\tfrac{1}{2}\Big\langle \sum_{n \ne 0} |n|^{-2s} a_n e^{inx}, \sum_{n \ne 0} a_n e^{inx} \Big\rangle_{\dot{H}_0^s} 
= -\tfrac{1}{2}\sum_{n \geq 1}  |a_n|^2.\] 

\noindent
The right hand side is exactly the expression appearing in the exponent in \eqref{White}.
By Theorem \ref{COUNTABLEADD}, $\mu$ is  countably additive
if and only if $B$ is  of trace class, i.e.
$ \sum_{n \ne 0} |n|^{2s} < \infty$.
Hence, $\bigcap_{s < -\frac{1}{2}} H^s$ is a natural space to work on.
Unfortunately, the results in \cite{BO3} and \cite{CCT} state that 
one can not have a local-in-time solution of \eqref{KDV} via the fixed point argument in $H^s$, $s < -\frac{1}{2}$.
Instead, we propose to work on $\ft{b}^s_{p, \infty} (\mathbb{T})$ defined in \eqref{Besov} for $sp < -1$
in view of Theorem \ref{THM:LWP2}.
Since $\ft{b}^s_{p, \infty}$ is not a Hilbert space, 
we need to go over  the basic theory of abstract Wiener spaces.

Recall the following definitions \cite{KUO}:
Given  a real separable Hilbert space $H$ with norm $\|\cdot \|$, 
let $\mathcal{F} $ denote the set of finite dimensional orthogonal projections $\mathbb{P}$ of $H$.
Then, define a cylinder set $E$ by  $E = \{ x \in H: \mathbb{P}x \in F\}$ where $\mathbb{P} \in \mathcal{F}$ 
and $F$ is a Borel subset of $\mathbb{P}H$,
and let $\mathcal{R} $ denote the collection of such cylinder sets.
Note that $\mathcal{R}$ is a field but not a $\s$-field.
Then, the Gauss measure $\mu$ on $H$ is defined 
by 
\[ \mu(E) = (2\pi)^{-\frac{n}{2}} \int_F e^{-\frac{\|x\|^2}{2}} dx  \]

\noindent
for $E \in \mathcal{R}$, where
$n = \text{dim} \mathbb{P} H$ and  
$dx$ is the Lebesgue measure on $\mathbb{P}H$.
It is known that $\mu$ is finitely additive but not countably additive in $\mathcal{R}$.

A seminorm $|||\cdot|||$ in $H$ is called measurable if for every $\eps>0$, 
there exists $\mathbb{P}_0 \in \mathcal{F}$ such that 
\[ \mu( ||| \mathbb{P} x ||| > \eps  )< \eps \]

\noindent
for $\mathbb{P} \in \mathcal{F}$ orthogonal to $\mathbb{P}_0$.
Any measurable seminorm  is weaker  than the norm of $H$,
and $H$ is not complete with respect to $|||\cdot|||$ unless $H$ is finite dimensional.
Let $B$ be the completion of $H$ with respect to $|||\cdot|||$
and denote by $i$ the inclusion map of $H$ into $B$.
The triple $(i, H, B)$ is called an abstract Wiener space.

Now, regarding $y \in B^\ast$ as an element of $H^\ast \equiv H$ by restriction,
we embed $B^\ast $ in $H$.
Define the extension of $\mu$ onto $B$ (which we still denote by $\mu$)
as follows.
For a Borel set $F \subset \R^n$, set
\[ \mu ( \{x \in B: ((x, y_1), \cdots, (x, y_n) )\in F\})
:= \mu ( \{x \in H: (\jb{x, y_1}_H, \cdots, \jb{x, y_n}_H )\in F\}),\]

\noindent
where $y_j$'s are in $B^\ast$ and $(\cdot , \cdot )$ denote the natural pairing between $B$ and $B^\ast$.
Let $\mathcal{R}_B$ denote the collection of cylinder sets
$ \{x \in B: ((x, y_1), \cdots, (x, y_n) )\in F \}$
in $B$.

\begin{theorem}[Gross \cite{GROSS}]
$\mu $ is countably additive in the $\s$-field generated by $\mathcal{R}_B$.
\end{theorem}

\noi
In the present context, let $H= L^2(\mathbb{T})$ and
$B=\ft{b}^s_{p, \infty} (\mathbb{T})$ for $sp < -1$. 
Then, we have 

\begin{proposition} \label{PROP:meas}
The seminorms  $\|\cdot\|_{\ft{b}^s_{p, \infty}}$ is measurable for $sp < -1$.

\end{proposition}

\noindent
Hence, $(i, H, B) = (i, L^2, \ft{b}^s_{p, \infty}) $ is an abstract Wiener space, 
and $\mu$ defined in \eqref{White} is countably additive in $\ft{b}^s_{p, \infty}$.
We present the proof of Proposition \ref{PROP:meas} at the end of this section.
It seems that the statement in Proposition \ref{PROP:meas} holds true for $sp = -1$
(c.f. Roynette \cite{ROY} for $p = 2$.)
However, we can choose $s$ and $p$ such that $sp < -1$ for our application,
and thus we will not discuss the endpoint case. 
It follows from the proof that $(i, L^2, \mathcal{F} L^{s, p}) $, where $ \mathcal{F}L^{s, p} = \ft{b}^s_{p, p}$,
is also an abstract Wiener space for $sp < -1$ (we need  a strict inequality in this case.)

Given an abstract Wiener space $(i, H, B)$, we have the following integrability result due to Fernique \cite{FER}.

\begin{theorem} [Theorem 3.1 in \cite{KUO}]
\label{THM:FER} 
Let $(i, H, B)$ be an abstract Wiener space.
Then, there exists $ c > 0$ such that $ \int_B e^{c \|x\|_B^2} \mu(d x) < \infty$.
Hence, there exists $ c' > 0$ such that $\mu ( \|x\|_B > K) \leq e^{-c'K^2}$.
\end{theorem}

\noi
In our context, if $sp < -1$, we have 
$ \mu \big(\|\phi\|_{\ft{b}^s_{p, \infty}(\mathbb{T})} \geq K, \phi \ \text{mean } 0 )  \leq e^{-c K^2}$ for some $c>0$.
With this estimate and Theorem \ref{THM:LWP2}, 
we can follow the argument in \cite{BO4} to prove Theorem \ref{THM:GWP2}.
We omit the details.
Also, see \cite{BT1}, \cite{OH3}, \cite{TZ1}, \cite{TZ2} for the details.

\begin{proof}[Proof of Proposition \ref{PROP:meas}]
We present the proof only for $2 < p < \infty$, which is the relevant case for our application.
We just point out that the proof for $p \leq 2$ is similar but simpler
(where one can use H\"older inequality in place of Lemma \ref{CL:decay} below.)
For $p = \infty$, see \cite{BO4}, \cite{BO5}, \cite{OH3}.

It suffices to show that
for given $\eps> 0$, 
there exists large $M_0$ such that 
\begin{equation*} 
 \mu \big(\|\mathbb{P}_{>M_0}\phi\|_{\ft{b}^s_{p, \infty}} > \eps) < \eps,
\end{equation*}

\noi 
where $\mathbb{P}_{>M_0}$is the projection onto the frequencies $|n| > M_0$.
In the following, 
write $\phi = \sum_{n\ne 0} g_n e^{inx}$, 
where $\{ g_n(\omega) \}_{n = 1}^\infty$ is a sequence of i.i.d. standard complex-valued Gaussian random variables
and $g_{-n} = \cj{g_n}$.
First, recall the following lemma.
\begin{lemma}[Lemma 4.7 in \cite{OHSBO}] \label{CL:decay}
Let $\{g_n\}$ be a sequence of i.i.d standard complex-valued Gaussian random variables.
Then, for $M$ dyadic and $\dl > 0$, we have
\[ \lim_{M\to \infty} M^{1-\dl} \frac{\max_{|n|\sim M } |g_n|^2}{ \sum_{|n|\sim M} |g_n|^2} = 0, \text{ a.s.} \]
\end{lemma}

Fix $K> 1$ and $\dl \in ( 0, \frac{1}{2})$ (to be chosen later.)
Then, by Lemma \ref{CL:decay} and Egoroff's theorem, 
there exists a set $E$ such that $\mu (E^c) < \frac{1}{2}\eps$
and the convergence in Lemma \ref{CL:decay} is uniform on $E$.
i.e. we can choose dyadic $M_0$ large enough such that 
\begin{equation} \label{A:decay}
\frac{\| \{g_n (\omega) \}_{|n| \sim M} \|_{L^{\infty}_n} }
{\| \{g_n (\omega) \}_{|n| \sim M} \|_{L^{2}_n} }
 \leq M^{-\dl}, 
\end{equation}

\noindent 
for  all $\omega \in E$  and dyadic $M > M_0$.
In the following, we will work only on $E$ and  drop `$\cap E$' for notational simplicity. 
However, it should be understood that all the events are under the intersection with $E$ 
so that  \eqref{A:decay} holds.

Let $\{\s_j \}_{j \geq 1}$ be a sequence of positive numbers such that $\sum \s_j = 1$,
and let $ M_j = M_0 2^j$ dyadic.
Note that $\s_j = C 2^{-\ld j} =C M_0^\ld M_j^{-\ld}$ for some small $\ld > 0$
(to be determined later.)
Then,  we have
\begin{align} \label{A:subadd} 
\mu \big(\|\mathbb{P}_{>M_0}\phi\|_{\ft{b}^s_{p, \infty}} > \eps)
& \leq \mu  \big( \| \{g_n \}_{|n| > M_0} \|_{{b}^s_{p, 1}} > \eps \big) \notag \\
& \leq \sum_{j = 0}^\infty \mu \big( \| \{\jb{n}^s g_n \}_{|n| \sim M_j} \|_{L_n^{p}}  > \s_j \eps \big),
\end{align}

\noindent
where ${b}^s_{p, 1}$ is defined in \eqref{Besov2}.
By interpolation and \eqref{A:decay}, we have  
\begin{align*}
\| \{ & \jb{n}^s g_n \}_{|n| \sim M_j} \|_{L_n^{p}} 
\sim M_j^{s} \| \{ g_n \}_{|n| \sim M_j} \|_{L_n^{p}} 
\leq  M_j^{s} \| \{ g_n \}_{|n| \sim M_j} \|_{L_n^{2}}^\frac{2}{p}
  \| \{ g_n \}_{|n| \sim M_j} \|_{L_n^{\infty}}^\frac{p-2}{p} \\
& \leq  M_j^{s} \| \{ g_n \}_{|n| \sim M} \|_{L_n^{2}}
\Bigg(\frac{  \| \{ g_n \}_{|n| \sim M_j} \|_{L_n^{\infty}} }
{\| \{ g_n \}_{|n| \sim M_j} \|_{L_n^{2}}} \Bigg)^\frac{p-2}{p}
\leq  M_j^{s -\dl \frac{p-2}{p}} \| \{ g_n \}_{|n| \sim M_j} \|_{L_n^{2}}
\end{align*}

\noindent
a. s. 
Thus, if we have $\|\{\jb{n}^s g_n \}_{|n| \sim M_j} \|_{L_n^{p}}  > \s_j \eps$,
 then we have
 $\| \{ g_n \}_{|n| \sim M_j} \|_{L_n^{2}} 
 \gtrsim R_j $
 where $R_j := \s_j \eps M_j^{-s+\dl \frac{p-2}{p}} $.
With $p = 2 + 2\theta$, we have 
$-s+\dl \frac{p-2}{p} = \frac{-sp + 2 \dl \theta}{2 + 2 \theta} > \frac{1}{2}$
by taking $\dl$ sufficiently close to $\frac{1}{2}$ since $-sp > 1$.
Then, by taking $\ld > 0$ sufficiently small,
$R_j = \s_j \eps M_j^{-s+\dl \frac{p-2}{p}} 
= C \eps M_0^\ld M_j^{-s+\dl \frac{p-2}{p}-\ld}
\gtrsim C \eps M_0^{\ld} M_j^{\frac{1}{2}+} $. 
By a direct computation in the polar coordinates, 
we have 
\begin{align*} 
\mu\big(  \| \{ g_n \}_{|n| \sim M_j} \|_{L_n^{2}}  \gtrsim R_j \big)
 \sim \int_{B^c(0, R_j)} e^{-\frac{1}{2}|g|^2} \prod_{|n| \sim M_j} dg_n
 \lesssim \int_{R_j}^\infty e^{-\frac{1}{2}r^2}  r^{2 \cdot \# \{|n| \sim M_j\} -1} dr.
\end{align*}

\noindent 
Note that the implicit constant in the inequality is 
$\s(S^{2 \cdot \# \{|n| \sim M_j\} -1})$, a surface measure of the 
$2 \cdot \# \{|n| \sim M_j\} -1$ dimensional unit sphere.
We drop it since $\s(S^n) = 2 \pi^\frac{n}{2} / \Gamma (\frac{n}{2}) \lesssim 1$.
By change of variable $t = M_j^{-\frac{1}{2}}r$, we have
$r^{2 \cdot \# \{|n| \sim M_j\} -2} \lesssim r^{4M_j} \sim M_j^{2M_j} t^{4M_j}.$
Since $t > M_j^{-\frac{1}{2}} R_j = C \eps M_0^{\ld} M_j^{0+}$, we have
\begin{equation*} 
M_j^{2M_j} = e^{2M_j \ln M_j} < e^{\frac{1}{8}M_jt^2} 
\ \ \text{ and } \ \ t^{4M_j} < e^{\frac{1}{8}M_jt^2}
\end{equation*}

\noindent
for $M_0$ sufficiently large.
Thus, we have
$r^{2 \cdot \# \{|n| \sim M_j\} -2} < e^{\frac{1}{4}M_jt^2} = e^{\frac{1}{4}r^2}$ 
for $ r > R.$
Hence,  we have
\begin{align} \label{A:highfreq1}
\mu\big(   \| \{ g_n \}_{|n| \sim M_j} \|_{L_n^{2}}  \gtrsim R_j \big)
\leq C \int_{R_j}^\infty e^{-\frac{1}{4}r^2} r dr  
\leq e^{-cR_j^2} = e^{-cC^2 M_0^{2\ld} M_j^{1+} \eps^2}.
\end{align}

\noindent
From \eqref{A:subadd} and \eqref{A:highfreq1}, we have
\begin{align*} 
\mu \big(\|\mathbb{P}_{>M_0}\phi\|_{\ft{b}^s_{p, \infty}} > \eps)
\leq \sum_{j =1}^\infty 
e^{-cC^2 M_0^{1+2\ld+} 2^{j+} \eps^2} \leq \tfrac{1}{2} \eps
\end{align*}

\noi
by choosing $M_0$ sufficiently large.
\end{proof}

\section{Local Well-Posedness in $\ft{b}^s_{p, \infty}$}
In this section, we prove Theorem \ref{THM:LWP2} via the fixed point argument.
In Subsection 4.1, we go over the previous local well-posedness theory of KdV
to motivate the definition of the Bourgain space $W^{s, b}_p$ with the weight,
adjusted to $\ft{b}^s_{p, \infty}$.
Then, we establish the basic linear estimates in Subsection 4.2.
Finally, we prove the crucial bilinear estimate in Subsection 4.3.

\subsection{Bourgain Space with a weight}

In \cite{KPV4}, Kenig-Ponce-Vega proved
\begin{equation} \label{KPVbilinear}
\| \dx(uv) \|_{X^{s, -\frac{1}{2}}} \lesssim \| u \|_{X^{s, \frac{1}{2}}} \| v \|_{X^{s, \frac{1}{2}}}, 
\end{equation}

\noi
for $s \geq -\frac{1}{2}$ under the mean 0 assumption on $u$ and $v$, where $X^{s, b}$ is defined in \eqref{Xsb}.
Their proof is based on proving the equivalent statement:
\begin{equation} \label{KPVdual}
\|B_{s}(f, g)\|_{L^2_{n, \tau} } \lesssim \|f\|_{L^2_{n, \tau}}\|g\|_{L^2_{n, \tau}}
\end{equation}

\noi
where $B_s (\cdot, \cdot)$ is defined by
\begin{equation} \label{BS}
B_s (f, g)(n, \tau) = \frac{1}{2\pi\jb{\tau-n^3}^{\frac{1}{2}}} \sum_{\substack{n_1 + n_2 = n\\n_1 \ne 0, n}}
\frac{|n| \jb{n}^s} {\jb{n_1}^s\jb{n_2}^s}
\intt_{\tau_1 + \tau_2 = \tau} 
\frac{f(n_1, \tau_1) g(n_2, \tau_2)d\tau_1}{\jb{\tau_1 - n_1^3}^\frac{1}{2}\jb{\tau_2 - n_2^3}^\frac{1}{2}}.
\end{equation}

\noi
One of the main ingredients is the observation due to Bourgain \cite{BO1}:
\begin{equation} \label{Walgebra}
n^3 - n_1^3 - n_2^3 = 3 n n_1 n_2, \ \text{for } n = n_1 + n_2,
\end{equation}

\noi
which in turn implies that 
\begin{equation} \label{Wmax}
\MAX := \max( \jb{\tau - n^3}, \jb{\tau_1 - n_1^3}, \jb{\tau_2 - n_2^3}) \gtrsim \jb{n n_1 n_2}
\end{equation}

\noi
for $n = n_1 + n_2$ and $\tau = \tau_1 + \tau_2$ with $n, n_1, n_2 \ne 0$.
Recall that \eqref{Wmax} implies that
\begin{equation} \label{KPVweight}
\frac{|n| \jb{n}^s} {\jb{n_1}^s\jb{n_2}^s} 
\frac{1}{\jb{\tau - n^3}^\frac{1}{2}\jb{\tau_1 - n_1^3}^\frac{1}{2}\jb{\tau_2 - n_2^3}^\frac{1}{2}} 
\lesssim \frac{|n| \jb{n}^s} {\jb{n_1}^s\jb{n_2}^s} \frac{1}{\MAX^\frac{1}{2}}
\lesssim 1
\end{equation}

\noi
for $ s \geq -\frac{1}{2}$.
Note that \eqref{KPVweight} is optimal,
for example, when $\jb{\tau-n^3} \sim \jb{3 n n_1 n_2}$
and $\jb{\tau_j - n_j^3} \ll \jb{3 n n_1 n_2}^{0+}$.
To exploit this along with the fact the free solution concentrates on the curve $\{ \tau = n^3\}$,
we define the weight $w(n, \tau)$ in the following.

For $k \in \mathbb{Z} \setminus \{0\}$ , let
\[ A_k = \{ (n, \tau): |n| \geq C, \jb{\tau - n^3 + 3n (n - k) k} \ll \jb{n}^\frac{1}{100} \}, \]

\noi
for some $C > 0$. With $\delta = 0+$ (to be determined later), let 
\begin{equation} \label{weight}
w( n, \tau) = 1 + \sum_{k \ne 0} \min( \jb{k}, \jb{n - k})^{\dl} \chi_{A_k}.
\end{equation}

\noi
Note that, for fixed $n$ and $\tau$, there are at most two values of $k$ 
such that $|(n-k)k +\frac{\tau - n^3}{3n}| \ll \jb{n}^{-1 + \frac{1}{100}}$. 
It follows from the definition that 
$w(n, \tau) \lesssim \max(1, \big(\frac{\jb{\tau-n^3}}{\jb{n}}\big)^{0+})
\leq \jb{\tau - n^3}^{0+}$.

Now, define the Bourgain space $W^{s, b}_p$ with the weight $w$ 
via the norm
\begin{equation}  \label{WBourgain}
\|u\|_{W^{s, b}_p} = \|\ft{u}\|_{\ft{W}^{s,b}_p} 
:= \| w \ft{u}\|_{\ft{X}^{s, b}_p} + \| \ft{u}\|_{\ft{Y}^{s, b-\frac{1}{2}}_p},
\end{equation}

\noi
where
\begin{align*} 
\begin{cases} \| f\|_{\ft{X}^{s, b}_p} 
:= \|\jb{n}^s \jb{\tau - n^3}^b f(n, \tau)\|_{b^0_{p, \infty}L^p_\tau} = \sup_j \|\jb{n}^s \jb{\tau - n^3}^b f(n, \tau)
\|_{L^p_{|n| \sim 2^j} L^p_\tau} \\
 \| f\|_{\ft{Y}^{s, b}_p} :=\|\jb{n}^s \jb{\tau - n^3}^b f(n, \tau)\|_{b^0_{p, \infty}L^1_\tau} 
 = \sup_j \|\jb{n}^s  \jb{\tau - n^3}^b f(n, \tau)
\|_{L^p_{|n| \sim 2^j} L^1_\tau}.
\end{cases}
\end{align*}

\noi
For our application, we set $b = \frac{1}{2}$.
Note that $Y^{s, 0}_p$ is introduced so that we have 
$W^{s, \frac{1}{2}}_p(\mathbb{T} \times [-T,T]) \subset C([-T, T]; \ft{b}^s_{p, \infty}(\mathbb{T}))$.
In the following, we take $p > 2$.

\subsection{Linear Estimates}

Let $S(t) = e^{-t \dx^3}$ and $\eta(t)$ be a smooth cutoff such that $\eta(t) = 1$ on $[-\frac{1}{2}, \frac{1}{2}]$ 
and $= 0$ for $|t| \geq 1$.

\begin{lemma} \label{LEM:linear1}
For any $s \in \mathbb{R}$, we have 
$\| \eta(t) S(t) u_0\|_{W^{s, \frac{1}{2}}_p} \lesssim \|u_0\|_{\ft{b}^s_{p, \infty}}$.
\end{lemma}

\begin{proof}
Recall that $w(n, \tau) \lesssim \jb{\tau-n^3}^{0+}$.
Noting that $(\eta(t) S(t) u_0)^\wedge (n, \tau) = \ft{\eta}(\tau-n^3) \ft{u_0}(n)$, 
we have
\begin{align*}
\|\eta(t) S(t) u_0\|_{W^{s, \frac{1}{2}}_p}
& \leq \sup_j \big\| \jb{n}^s  \| \jb{\tau -n^3}^{\frac{1}{2}+}\ft{\eta}(\tau-n^3)\|_{L^p_\tau} 
|\ft{u_0}(n)|\big\|_{L^p_{|n|\sim 2^j}}\\
& \hphantom{X} + \sup_j \big\| \jb{n}^s  \|\ft{\eta}(\tau-n^3) \|_{L^1_\tau} |\ft{u_0}(n)|\big\|_{L^p_{|n|\sim 2^j}}
 \leq C_\eta \|u_0\|_{\ft{b}^s_{p, \infty}},
\end{align*}

\noi
where $C_\eta = \|\jb{\tau}^{\frac{1}{2}+}\ft{\eta}(\tau)\|_{L^p_\tau}  + \|\ft{\eta}\|_{L^1} < \infty$.
\end{proof}

Now, we estimate the Duhamel term.
By the standard computation \cite{BO1}, we have
\begin{align} \label{Duhamel}
 \int_0^t S(t-t') F(x, t') dt'
& = -i \sum_{k= 1}^\infty \frac{i^k t^k}{k!}  
\sum_{n \ne 0} e^{i(nx + n^3t)} \int \eta(\ld - n^3) \ft{F}(n, \ld) d\ld \notag \\
& \hphantom{X}+ i  \sum_{n \ne 0} e^{inx} \int 
\frac{\big(1-\eta\big)(\ld - n^3)}{\ld - n^3} e^{i \ld t} \ft{F}(n, \ld) d \ld \notag \\
& \hphantom{X}+ i  \sum_{n \ne 0} e^{i(nx + n^3 t)} \int 
\frac{\big(1-\eta\big)(\ld - n^3)}{\ld - n^3}  \ft{F}(n, \ld) d \ld \notag \\
& =:  \mathcal{N}_1(F)(x, t) + \mathcal{N}_2(F)(x, t)+ \mathcal{N}_3(F)(x, t).
\end{align}

\begin{lemma} \label{LEM:linear2}
For any $s \in \mathbb{R}$, we have 
\[ \| \eta (t) \mathcal{N}_1 (F)\|_{W^{s, \frac{1}{2}}_p}, \ 
\| \mathcal{N}_2(F) \|_{W^{s, \frac{1}{2}}_p},
\| \eta (t) \mathcal{N}_3 (F)\|_{W^{s, \frac{1}{2}}_p} 
\lesssim \| F \|_{W^{s, -\frac{1}{2}}_p}.\]

\end{lemma}

\begin{proof}
Recall that $w(n, \tau) \lesssim \jb{\tau-n^3}^{0+}$.
Let $\eta_k(t) = t^k \eta(t)$.
First, note that $|\eta_k(t)| \leq |\eta(t)|$ since $\eta(t) =0 $ for $|t| \geq 1$.
Moreover, 
by Hausdorff-Young and H\"older inequalities, 
we have $\| \jb{\tau }^{\frac{1}{2}+} 
\ft{\eta_k} (\tau )\|_{L^p_\tau} \leq \|\eta_k\|_{H^{\frac{1}{2}+}_t}
\leq \|\eta_k\|_{H^{1}_t} \lesssim 1 + k$.
Then, by Minkowski integral inequality, we have
\begin{align*}
\| \eta (t) \mathcal{N}_1 (F)\|_{X^{s, \frac{1}{2}}_p}
\leq C_\eta \sup_j \Big\| \jb{n}^s \int \eta( \ld - n^3) |\ft{F}(n, \ld)| d \ld \Big\|_{L^p_{|n| \sim 2^j}}
\lesssim C_\eta \|F\|_{Y^{s, -1}_p},
\end{align*}

\noi
where 
$C_\eta = \sup_n \sum_{k = 1}^\infty \frac{1}{k!} \| \jb{\tau -n^3}^{\frac{1}{2}+} 
\ft{\eta_k} (\tau - n^3)\|_{L^p_\tau}
\leq \sum_{k = 1}^\infty \frac{\| \jb{\tau }^{\frac{1}{2}+} 
\ft{\eta_k} (\tau )\|_{L^p_\tau}}{k!}
\lesssim \sum_{k = 1}^\infty \frac{1+k}{k!} < \infty.$
Similarly, 
we have
\begin{align*}
\| \eta (t) \mathcal{N}_1 (F)\|_{Y^{s, 0}_p}
\leq C'_\eta \sup_j \Big\| \jb{n}^s \int \eta( \ld - n^3) |\ft{F}(n, \ld)| d \ld \Big\|_{L^p_{|n| \sim 2^j}}
\lesssim C'_\eta \|F\|_{Y^{s, -1}_p},
\end{align*}

\noi
where 
$C'_\eta = \sup_n \sum_{k = 1}^\infty \frac{1}{k!} \| 
\ft{\eta_k} (\tau - n^3)\|_{L^1_\tau}$.
Now, note that
\[\sup_n \|\ft{\eta_k} (\tau - n^3)\|_{L^1_\tau} 
\leq \sup_n \| \jb{\tau-n^3}^{-\frac{1}{p'}-} \|_{L^{p'}_\tau}
\| \jb{\tau -n^3}^{\frac{1}{p'}+} 
\ft{\eta_k} (\tau - n^3)\|_{L^p_\tau}
\lesssim 1 + k,\]

\noi
since $\frac{1}{p'}+ = 2 +$.
Hence, we have $C'_\eta < \infty$ as before.
 
For $|\tau -n^3| \gtrsim 1$, we have $|\tau-n^3|\sim \jb{\tau - n^3}$.
Thus, 
we have $\ft{\mathcal{N}_2(F)}(n, \tau) \lesssim \jb{\tau - n^3}^{-1} \ft{F}(n, \tau)$.
Then, by monotonicity (i.e. $\|f\|_{\ft{W}^{s, \frac{1}{2}}_p} \leq \|g\|_{\ft{W}^{s, \frac{1}{2}}_p}$ 
for $|f| \leq |g|$), we have
$\| \mathcal{N}_2(F) \|_{W^{s, \frac{1}{2}}_p}
\lesssim \| F \|_{W^{s, -\frac{1}{2}}_p}.$

Lastly, by Minkowski integral inequality with $w(n, \tau) \lesssim \jb{\tau-n^3}^{0+}$, we have 
\begin{align*}
\| \eta(t) \mathcal{N}_3(F) \|_{X^{s, \frac{1}{2}}_p}
&= \sup_j \big\| \jb{n}^s \jb{\tau-n^3}^{\frac{1}{2}+} \ft{\eta}(\tau-n^3) 
\int \frac{1 -\eta(\ld - n^3)}{\ld - n^3} |\ft{F}(n, \ld)| d \ld \big\|_{L^p_{|n|\sim 2^j} L^P_\tau} \\
&\leq C_\eta \|F\|_{Y^{s, -1}_p},
\end{align*}

\noi
where $C_\eta = \sup_n \| \jb{\tau-n^3}^{\frac{1}{2}+}\ft{\eta}(\tau-n^3) \|_{L^p_\tau}
= \| \jb{\tau}^{\frac{1}{2}+}\ft{\eta}(\tau) \|_{L^p_\tau} < \infty$.
Similarly, we have 
\begin{align*}
\| \eta (t) \mathcal{N}_3 (F)\|_{Y^{s, 0}_p}
\lesssim C'_\eta \|F\|_{Y^{s, -1}_p},
\end{align*}

\noi
where $C'_\eta = \sup_n \| \ft{\eta}(\tau-n^3) \|_{L^1_\tau}
= \|\ft{\eta}\|_{L^1_{\tau}} < \infty$.
\end{proof}

\subsection{Bilinear estimate}

 By expressing \eqref{KDV} in the integral formulation, 
we see that $u$ is a solution to \eqref{KDV} for $|t| \leq T \ll 1$
if and only if
$u$ satisfies
\begin{align*}
u(t) :& = \Phi^t_{u_0}(u) \notag \\
& =\eta(t) S(t) u_0 + 
 \eta (t) \mathcal{N}_1 (\eta_{_{2T}}F(u)) (t)
+ \mathcal{N}_2(\eta_{_{2T}}F(u)) (t)
+ \eta (t) \mathcal{N}_3 (\eta_{_{2T}}F(u))(t),
\end{align*}

\noi
where $F (u) = -u u_x$ and $\eta_{_{2T}}(t) = \eta(t/2T)$, 
i.e. $\eta_{_{2T}}(t) \equiv 1$ for $|t| \leq T$.
In this subsection, we prove the crucial bilinear estimate
so that $\Phi^t_{u_0}(\cdot)$ defined above is a contraction 
on a ball in
$W^s_p(\mathbb{T} \times [-T,T]) \subset C([-T, T], \ft{b}^s_{p, \infty}(\mathbb{T}))$
for $T$ sufficiently small.

\begin{proposition} \label{PROP:Wbilinear1}
Assume that $u$ and $v$ have the spatial means 0 for all $t \in \mathbb{R}$.
Then, there exist $s = -\frac{1}{2}+$, $p = 2+$ with $s p < -1$,
and $\theta > 0$
such that
\begin{equation} \label{Xbilinear1}
\| \eta_{_{2T}} \dx (uv) \|_{W^{s, -\frac{1}{2}}_p} \lesssim 
T^\theta \| u\|_{W^{s, \frac{1}{2}}_p}
\| v\|_{W^{s, \frac{1}{2}}_p}.
\end{equation}

\end{proposition}

Before proving Proposition \ref{PROP:Wbilinear1}, 
we present some lemmata.

\begin{lemma} [Ginibre-Tsutsumi-Velo \cite{GTV}, Lemma 4.2 ]\label{LEM:GTV}
Let $ 0 \leq \al \leq \beta$ and $ \al + \beta > \frac{1}{2}$.
Then, we have 
\[ \int \jb{\tau}^{-2\al} \jb{\tau -a }^{-2\beta} d \tau  \lesssim \jb{a}^{-\g},\]

\noindent
where $\g = 2\al - [1 - 2\beta]_+$ with 
$[x]_+ = x$ if $x>0$, $= \eps > 0$ if $x = 0$, and $ = 0$ if $x < 0$. 
\end{lemma}

\begin{lemma} \label{LEM:Psum}
For $l_1 + 2l_2>1$ with $l_1, l_2 > 0$, there exists $c > 0$ such that for all $n\ne 0$ and $\ld \in \R$, we have
\begin{equation} \label{Psum1}
 \sum_{n_1 \ne 0, n} \frac{1}{\jb{n_1}^{l_1}}\frac{1}
{\jb{\ld + n_1(n-n_1)}^{l_2}} < c.
\end{equation}

\end{lemma}

\begin{proof}
  When $l_2 = 0$, \eqref{Psum1} is clear. When $l_1 = 0$, \eqref{Psum1} follows from Lemma 5.3 in \cite{KPV7}.
Thus, we assume $l_1, l_2 > 0$ in the following.
Since $l_1 + 2l_2>1$, there exists $\eps > 0$ such that $l_1 + 2l_2 - 3\eps \geq 1$.

If $P_{n, \ld} (n_1) := \ld + n_1(n-n_1)$ has two real roots, i.e. $P_{n, \ld} (n_1)= -(n_1 -r_1)(n_1 -r_2)$, then
there are at most 6 values of $n_1$ such that $|n_1 -r_j| \leq 1$.
For the remaining values of $n_1$, we have
$\jb{P_{n, \ld} (n_1)} > \frac{1}{4} \prod_{j = 1}^2 \jb{n_1 - r_j}$.
Then, \eqref{Psum1} follows from H\"older inequality with $p = (l_1 - \eps)^{-1}$ and $q = (l_2 - \eps)^{-1}$, we have 
\[ \text{LHS of } \eqref{Psum1} 
\lesssim \Big( \sum_{n_1} {\jb{n_1}^{-p l_1} }\Big)^\frac{1}{p}
\prod_{j = 1}^2\Big(\sum_{n_1}  {\jb{ n_1 -r_j}^{-ql_2}} \Big)^\frac{1}{q}
< c < \infty,\]

\noindent
since $ p l_1 > 1$ and $q l_2 > 1$.

If $P_{n, \ld} (n_1)$ has only one  or no real root, then we have 
$| P_{n, \ld} (n_1)| \geq (n_1 -\frac{1}{2}n)^2$ for all $n_1 \in \mathbb{Z}$.
Then, by H\"older inequality with $p = (l_1 - \eps)^{-1}$ and $q = (2l_2 - 2\eps)^{-1}$, we have
\[ \text{LHS of } \eqref{Psum1} 
\leq \Big( \sum_{n_1}  {\jb{n_1}^{-p l_1} }\Big)^\frac{1}{p}
\Big(\sum_{n_1}  {\jb{ (n_1 -\tfrac{1}{2}n)^2}^{-ql_2}} \Big)^\frac{1}{q}
< c < \infty,\]

\noindent
since $ p l_1 > 1$ and $2q l_2 = \frac{l_2}{l_2 - \eps}> 1$.

\end{proof}

\noindent
Lastly, recall the following lemma from \cite[(7.50) and Lemma 7.4]{CKSTT4}.

\begin{lemma} \label{LEM:closetocurve}
Let 
\begin{equation*} 
 \Omega(n) = \{ \eta \in \R : \eta = -3 n n_1 n_2 + o(\jb{n n_1 n_2}^\frac{1}{100}) \text{ for some } n_1 \in \mathbb{Z} 
\text{ with } n = n_1 + n_2 \}. 
\end{equation*}

\noindent
Then, we have
\begin{equation} \label{closetocurve}
 \int \jb{\tau - n^3}^{-\frac{3}{4}} \chi_{\Omega(n)} (\tau - n^3) d \tau \lesssim 1. 
 \end{equation}
\end{lemma}

\noi
Note that 
\eqref{closetocurve} is stated with $\jb{\tau - n^3}^{-1}$ in \cite{CKSTT4}.
However, by examining the proof of Lemma 7.4 in \cite{CKSTT4}, 
one immediately sees that \eqref{closetocurve} is valid 
with $\jb{\tau - n^3}^{-\alpha}$ for any $\alpha > \frac{2}{3}+\frac{1}{100}$.

\begin{proof} [Proof of Proposition \ref{PROP:Wbilinear1}]
In the proof, we use $(n, \tau)$, $(n_1, \tau_1)$, and $(n_2, \tau_2)$ to denote the Fourier variables
for $uv$, $u$, and $v$, respectively.
i.e. we have $n = n_1 + n_2$ and $\tau = \tau_1 + \tau_2$
Moreover, by the mean 0 assumption on $u$ and $v$ and 
by the fact that we have $\dx (uv)$ on the left hand side of \eqref{Xbilinear1}, 
we assume  $n, n_1, n_2 \ne 0$ in the following.

First, we prove
\begin{equation} \label{Xbilinear2}
\|  \dx (uv) \|_{W^{s, -\frac{1}{2}}_p} \lesssim 
 \| u\|_{W^{s, \frac{1}{2}}_p}
\| v\|_{W^{s, \frac{1}{2}}_p}.
\end{equation}

\noi
i.e. we first prove \eqref{Xbilinear1} with no gain of $T^\theta$.
Then, it suffices to show
\begin{equation} \label{Wbilinear}
 \| B(f, g)(n, \tau) \|_{\ft{W}_p^{0, -\frac{1}{2}}} 
\lesssim \|f\|_{b^0_{p, \infty} L^p_\tau}\|g\|_{b^0_{p, \infty}L^p_\tau}.
\end{equation}

\noi
where $B(\cdot, \cdot)$ is defined by
\begin{equation*}
B(f, g)(n, \tau) = 
\frac{1}{2\pi} \sum_{n_1 + n_2 = n}
\frac{|n| \jb{n}^s} {\jb{n_1}^s\jb{n_2}^s}
\intt_{\tau_1 + \tau_2 = \tau} 
\frac{f(n_1, \tau_1) g(n_2, \tau_2)d\tau_1}{\prod_{j = 1}^2 w(n_j, \tau_j)\jb{\tau_j - n_j^3}^\frac{1}{2}}.
\end{equation*}

\noi
Let $\MAX := \max(\jb{\tau - n^3}, \jb{\tau_1 - n_1^3}, \jb{\tau_2 - n_2^3}) $.
Then, by \eqref{Walgebra}, we have  $\MAX \gtrsim \jb{n n_1 n_2}$.

\medskip

\noi
$\bullet$ {\bf PART 1:}
First, we consider the $\ft{X}^{0, -\frac{1}{2}}_p$ part 
of the $\ft{W}_p^{0, -\frac{1}{2}}$ norm on the left hand side of \eqref{Wbilinear}.

\noi
$\bullet$ {\bf Case (1):} $\MAX = \jb{\tau - n^3}$.
Without loss of generality, assume $|n_1| \geq |n_2|$.
For fixed $n\ne0$ and $\tau$, let $\ld = \frac{\tau - n^3}{3n}$ and define 
\begin{align*}
B_{n, \tau} = \{ n_1 & \in \mathbb{Z}: |n_1 - r_j| \geq 1, j = 1, 2\\
& r_j \ \text{ is a real root of } P_{n, \ld}(n_1) := \ld + n_1 (n - n_1) \\
& \text{or } r_j = \frac{1}{2}n \ \text{ if no real root} \}.
\end{align*}

\noi
On $B_{n, \tau}$, we have
\begin{equation} \label{Paway1}
\jb{\tau - n^3 + 3 n n_1 n_2} \gtrsim \jb{n} \jb{\ld + n_1(n-n_1)}.
\end{equation}

\noi
$\circ$ Subcase (1.a): On $B_{n, \tau}^c$. \
For $s > -\frac{1}{2}$, we have 
\begin{equation} \label{upperbd1}
\frac{|n|\jb{n}^s}{\jb{n_1}^s\jb{n_2}^s} \frac{1}{\MAX^\frac{1}{2}}
\lesssim \frac{1}{\jb{n_2}^{\frac{1}{2} + s}}.
\end{equation}

\noi
By Lemma \ref{LEM:GTV},  we have 
\[ \| \jb{\tau_1 - n_1^3}^{-\frac{1}{2}} \jb{\tau_2 - n_2^3}^{-\frac{1}{2}} \|_{L^{p'}_{\tau_1}}
\lesssim \jb{\tau - n^3 + 3 n n_1 n_2}^{-1+\frac{1}{p'}}.\]

\noi
Note that for fixed $n$ and $\tau$ there are at most four values of $n_1 \in B^c_{n, \tau}$.
i.e. the summation over $n_1$ can be replace by the $L^{p}_{n_1}$ norm.
Then, by H\"older inequality, we have
\begin{align*}
\text{LHS of \eqref{Wbilinear} } & \lesssim 
\sup_j \bigg\| \sum_{n = n_1+n_2 } \frac{w(n, \tau)}{\jb{n_2}^{\frac{1}{2}+s}}
\intt_{\tau = \tau_1 + \tau_2} \frac{f(n_1, \tau_1) g(n_2, \tau_2) d\tau_1}
{\jb{\tau_1 - n_1^3}^\frac{1}{2} \jb{\tau_2 - n_2^3}^\frac{1}{2}}
\bigg\|_{L^p_{|n| \sim 2^j} L^p_\tau} \\
& \lesssim 
\sup_j \big\|  \frac{w(n, \tau)}{\jb{n_2}^{\frac{1}{2}+s}}
\|f(n_1, \cdot)\|_{L^p_\tau}\|g(n_2, \cdot)\|_{L^p_\tau} 
\big\|_{L^p_{|n| \sim 2^j} L^p_{n_1}}.
\end{align*}

\noindent
Note that $w(n, \tau) \lesssim \jb{n_2}^\dl$ since $|n_1| \geq |n_2|$.  
If $|n_1| \gg |n_2|$ and $|n| \sim 2^j$, then we have $|n_1| \sim 2^{k}$ where  $|k - j| \leq 5$.
\begin{align*}
\text{LHS of \eqref{Wbilinear} } & \lesssim 
\sup_j \bigg(
\sum_{|k - j| \leq 5} \sum_{|n_1| \sim 2^k }
\sum_{l = 0}^\infty \sum_{|n_2| \sim 2^l}
  \jb{n_2}^{(-\frac{1}{2}-s +\dl)p}
\|f(n_1, \cdot)\|^p_{L^p_\tau}\|g(n_2, \cdot)\|^p_{L^p_\tau} 
 \bigg)^\frac{1}{p} \\
& \lesssim  
\sum_{l = 0}^\infty 2^{(-\frac{1}{2}-s +\dl)p \, l}
  \sup_k \|f\|_{L^p_{|n| \sim 2^k } L^p_\tau}
\sup_l \|g \|_{L^p_{|n| \sim 2^l}L^p_\tau} 
\lesssim  
\|f\|_{b^0_{p, \infty}L^p_\tau}\|g\|_{b^0_{p, \infty}L^p_\tau},
 \end{align*}

\noi
by taking $\dl > 0$ sufficiently small such that $ -\frac{1}{2}-s +\dl< 0$.
Similarly, if $|n_1| \sim |n_2|$ and $|n_2| \sim 2^l$, then we have $|n_1| \sim 2^{k}$ where  $|k - l| \leq 5$.
\begin{align*}
\text{LHS of \eqref{Wbilinear} } & \lesssim 
 \bigg(
\sum_{l = 0}^\infty \sum_{|k - l| \leq 5} \sum_{|n_1| \sim 2^k }\sum_{|n_2| \sim 2^l}
  \jb{n_2}^{(-\frac{1}{2}-s +\dl)p}
\|f(n_1, \cdot)\|^p_{L^p_\tau}\|g(n_2, \cdot)\|^p_{L^p_\tau} 
 \bigg)^\frac{1}{p} \\
& \lesssim  
\sum_{l = 0}^\infty 2^{(-\frac{1}{2}-s +\dl)p \, l}
  \sup_k \|f\|_{L^p_{|n| \sim 2^k } L^p_\tau}
\sup_l \|g \|_{L^p_{|n| \sim 2^l}L^p_\tau} 
\lesssim  
\|f\|_{b^0_{p, \infty }L^p_\tau}\|g\|_{b^0_{p, \infty } L^p_\tau}.
 \end{align*}

\noi
$\circ$ Subcase (1.b): On $B_{n, \tau}$. \
In this case, we have \eqref{Paway1}.
Also, recall that $w(n, \tau) \lesssim \jb{\tau - n^3}^{0+}$.
Moreover, 
$\jb{\tau-n^3}^{0+} \lesssim \max(  \jb{n}, \jb{n_2}, \jb{\tau - n^3 + 3 n n_1 n_2})^{0+}$
since either $\jb{\tau-n^3} \gg |n n_1 n_2|$ or $\jb{\tau-n^3} \lesssim |n n_1 n_2| \lesssim \max(  \jb{n}^3, \jb{n_2}^3)$.
In particular, by \eqref{Paway1}, we have
\begin{equation} \label{wbound1}
w(n, \tau)  \lesssim (\jb{n_2}\jb{\tau - n^3 + 3 n n_1 n_2})^{0+}.
\end{equation}

\noi
By applying H\"older inequality and proceeding as before, we have
\begin{align*}
\text{LHS of \eqref{Wbilinear} } \lesssim M \,
\sup_j \big\|  \jb{n_2}^{0-}
\|f(n_1, \cdot)\|_{L^p_\tau}\|g(n_2, \cdot)\|_{L^p_\tau} 
\big\|_{L^p_{|n| \sim 2^j} L^p_{n_1}}
\lesssim M \|f\|_{b^0_{p, \infty}L^p_\tau}\|g\|_{b^0_{p, \infty}L^p_\tau},
\end{align*}

\noi
where
\[M = \sup_{n, \tau} \bigg\| \frac{w(n, \tau)} {\jb{n_2}^{\frac{1}{2}+s -} \jb{\tau - n^3 + 3 n n_1 n_2}^{1-\frac{1}{p'}}} \bigg\|_{L^{p'}_{n_1} }.\]

\noi
Thus, it remains to show that $M < \infty.$
By \eqref{wbound1}, \eqref{Paway1}, and Lemma \ref{LEM:Psum}, we have 
\begin{align*}
M^{p'} \lesssim 
\sup_{n, \tau} \frac{1}{\jb{n}^{p' -1-} }
\sum_{n_2}  \frac{1} {\jb{n_2}^{(\frac{1}{2}+s -)p'} \jb{\ld + n_1(n-n_1)}^{p'-1-}}
< \infty,
\end{align*}

\noi
since $(\frac{1}{2}+s -)p' + 2(p'-1)- > 1$ for $p = 2+ < 4$ and $sp = -1-$.
 
\medskip

Now, assume $\MAX = \jb{\tau_2 - n_2^3}$. 
 By symmetry, this takes care of the case when $\MAX = \jb{\tau_1 - n_1^3}$.
  Note that we have $w(n, \tau) \lesssim \jb{\tau - n^3}^{0+}$ by a crude estimate.
 Thus,  by duality, it suffices to show
\begin{align} \label{DUAL1}
&\sum_{l = 0}^\infty \bigg\| \sum_n 
 \frac{|n|\jb{n}^s}{\jb{n_1}^s\jb{n_2}^s} \frac{1}{w(n_2, \tau_2)\jb{\tau_2 - n_2^3}^\frac{1}{2}}
\int  \frac{f(n_1, \tau_1)   h(n, \tau) d\tau }{\jb{\tau_1 - n_1^3}^\frac{1}{2}\jb{\tau - n^3}^{\frac{1}{2}-}}
 \bigg\|_{L^{p'}_{|n_2| \sim 2^l} L^{p'}_{\tau_2}}\\
 \lesssim & \sup_k \|f\|_{L^{p}_{|n_1| \sim 2^k} L^{p}_{\tau_1}}
\sum_{j=0}^\infty \|h\|_{L^{p'}_{|n| \sim 2^j} L^{p'}_{\tau}} . \notag
\end{align}

For fixed $n_2\ne0$ and $\tau_2$, 
let $\ld = \frac{\tau_2 - n_2^3}{3n_2}$ and define
\begin{align*}
\wt{B}_{n_2, \tau_2} = \{ n & \in \mathbb{Z}: |n - r_j| \geq 1, j = 1, 2\\
& r_j \ \text{ is a real root of } P_{n_2, \ld}(n) := \ld + n (n_2 - n) \\
& \text{or } r_j = \frac{1}{2}n_2 \ \text{ if no real root} \}.
\end{align*}

\noi

\noi
On $\wt{B}_{n_2, \tau_2}$, we have
\begin{equation} \label{away2}
\jb{\tau_2 - n_2^3 - 3 n n_1 n_2} \gtrsim \jb{n_2} \jb{\ld + n(n_2-n)}. 
\end{equation}

\noi
Also, note that 
$w(n_2, \tau_2) \lesssim \min(\jb{n}^\dl, \jb{n_1}^\dl)$ 
on $\wt{B}^c_{n_2, \tau_2}$.

\noi 
$\bullet$ {\bf Case (2):} $\MAX = \jb{\tau_2 - n_2^3}$ and $|n_1| \gtrsim |n_2|$. 
In this case, we have
\begin{equation} \label{upperbd2}
\frac{|n|\jb{n}^s}{\jb{n_1}^s\jb{n_2}^s} \frac{1}{\MAX^\frac{1}{2}}
\lesssim \frac{1}{\jb{n_2}^{\frac{1}{2} + s}}.
\end{equation}

\noi
$\circ$ Subcase (2.a): On $\wt{B}^c_{n_2, \tau_2}$.

First, suppose $\jb{\tau_2 - n_2^3 - 3n n_1 n_2} \gtrsim \jb{n_2}^\frac{1}{100}$.
 Thus, by Lemma \ref{LEM:GTV},  we have
 \begin{equation} \label{GTV2a} \|\jb{\tau_1 - n_1^3}^{-\frac{1}{2}+\alpha} \jb{\tau - n^3}^{-\frac{1}{2}+} \|_{L^p_\tau}
 \lesssim     \jb{\tau_2 - n_2^3 - 3n n_1 n_2}^{-\frac{1}{2} + \al+} 
 \lesssim \jb{n_2}^{\frac{-1}{100}(\frac{1}{2} - \al)+}
 \end{equation}

\noindent
for $\alpha > 0$.
Then,  by H\"older inequality in $\tau$ followed 
by Young and H\"older inequalities, we have
\begin{align*}
\bigg\| \int  &  \frac{f(n_1, \tau_1)   h(n, \tau) d\tau }{\jb{\tau_1 - n_1^3}^\frac{1}{2}\jb{\tau - n^3}^{\frac{1}{2}-}}
 \bigg\|_{ L^{p'}_{\tau_2}}
  \lesssim \jb{n_2}^{\frac{-1}{100}(\frac{1}{2} - \al)+} 
 \Big\|\frac{f(n_1, \tau_1)}{\jb{\tau_1 - n_1^3}^\alpha}   h(n, \tau) \Big\|_{ L^{p'}_{\tau_2,  \tau}} \\
& \leq  \jb{n_2}^{\frac{-1}{100}(\frac{1}{2} - \al)+} 
\| \jb{\tau_1 - n_1^3}^{-\alpha}\|_{L^\frac{p}{p-2}_{\tau_1}} \|f(n_1, \cdot)\|_{L^{p}_{\tau_1}}    
\| h(n, \cdot) \|_{ L^{p'}_{  \tau}} 
\end{align*}

\noi
for fixed $n$ and $n_1$.
By choosing $\al > \frac{p-2}{p} = 0+$, we have 
$\| \jb{\tau_1 - n_1^3}^{-\alpha}\|_{L^\frac{p}{p-2}_{\tau_1}} < C< \infty$, independently of $n_1$.

Note that if $|n| \sim 2^j$ and $|n_2| \sim 2^l$, 
then we have $|n_1| \sim 2^k$ where $|k - j| \leq 5$ or $|k - l| \leq 5$
since $n = n_1 + n_2$ and $|n_1| \geq |n_2|$.
As in Subcase (1.a), 
for fixed $n_2$ and $\tau_2$ there are at most four values of $n \in \wt{B}^c_{n_2, \tau_2}$.
i.e. the summation over $n$ can be replace by the $L^{p'}_n$ norm.
By H\"older inequality in $n_2$ after switching the order of summations, 
\begin{align} \label{case2a}
\text{LHS of \eqref{DUAL1} }
& \lesssim 
\sum_{l = 0}^\infty \big\| 
 \jb{n_2}^{-\frac{1}{2} -s  - \frac{1}{100}(\frac{1}{2} - \al)+} 
\|f(n_1, \cdot)\|_{L^{p}_{\tau_1}}    
\| h(n, \cdot) \|_{ L^{p'}_{ \tau}} 
 \big\|_{L^{p'}_{|n_2| \sim 2^l} L^{p'}_{n}}    \notag \\
& \lesssim \Big(\sum_{l = 0}^\infty (2^l)^{0-}\Big)
\sup_{l } \bigg( \sum_{j = 0 }^\infty \sum_{|n|\sim 2^j}
 \|\jb{n_2}^{-\frac{1}{2} -s - \frac{1}{100}(\frac{1}{2} - \al)+} \|^{p'}_{L^\frac{p}{p-2}_{n_2}} \\
 & \hphantom{XXXXX}\times
\|f(n -n_2, \cdot)\|^{p'}_{L^p_{|n_2|\sim 2^l} L^{p}_{\tau_1}}    
\| h(n, \cdot) \|^{p'}_{ L^{p'}_{ \tau}} 
 \bigg)^\frac{1}{p'}   \notag \\
&  \lesssim 
 \wt{M} \|f\|_{b^0_{p, \infty}L^p_\tau}\|h\|_{b^0_{p', 1}L^{p'}_\tau}, \notag
 \end{align}

\noi
where $\wt{M}=\|\jb{n_2}^{-\frac{1}{2} -s - \frac{1}{100}(\frac{1}{2} - \al)+} \|_{L^\frac{p}{p-2}_{n_2}}
< \infty$, since  
$\big(\frac{1}{2} +s + \frac{1}{100}(\frac{1}{2} - \al)-\big) \frac{p}{p-2} > 1$
for $p < \frac{2-}{1-\frac{1}{100}+} \sim \frac{200-}{99} $ with $sp < -1$.
Note that we did not make use of $w(n_2, \tau_2)$ in this case.

\medskip

Now, 
suppose $\jb{\tau_2 - n_2^3 - 3n n_1 n_2} \ll \jb{n_2}^\frac{1}{100}$.
In this case, we can not expect any contribution from $\jb{\tau_2 - n_2^3 - 3n n_1 n_2}$ in \eqref{GTV2a}.
However, as long as we gain a small power of $\jb{n_2}$ in the denominator of LHS of \eqref{DUAL1}, 
we can proceed as before.
 Note that  $w(n_2, \tau_2) \sim \jb{n}^\dl$ since $|n_1| \gtrsim |n_2|$
 implies $|n| \lesssim |n_1|$.
If $|n_2| \lesssim |n|^{100}$, then we have $w(n_2, \tau_2) \gtrsim \jb{n_2}^\frac{\dl}{100}$.
Otherwise, we have $|n_1| \gtrsim |n_2| \gg |n|^{100}$.
Then, instead of \eqref{upperbd2}, we have
\begin{equation*} 
\frac{|n|\jb{n}^s}{\jb{n_1}^s\jb{n_2}^s} \frac{1}{\MAX^\frac{1}{2}}
\lesssim \frac{1}{\jb{n_1}^{(\frac{1}{2} + s)\frac{99}{100}} \jb{n_2}^{\frac{1}{2} + s}}
\lesssim \frac{1}{ \jb{n_2}^{\frac{1}{2} + s + \eps} }
\end{equation*}

\noi
for some $\eps = 0+$.
Hence, we obtain a small power of $\jb{n_2}$ in either case.

\smallskip

\noi
$\circ$ Subcase (2.b): On $\wt{B}_{n_2, \tau_2}$. \
In this case, we have \eqref{away2}.
As in Subcase (2.a), choose small $\al > \frac{p-2}{p} = 0+$.
By H\"older inequality with \eqref{GTV2a} and \eqref{away2},
we have
\begin{align*}
 \int  &  \frac{f(n_1, \tau_1)   h(n, \tau) d\tau }{\jb{\tau_1 - n_1^3}^\frac{1}{2}\jb{\tau - n^3}^{\frac{1}{2}-}}
   \lesssim 
   \jb{n_2}^{-\frac{1}{2} + \al+} 
   \jb{\ld + n(n_2-n)}^{-\frac{1}{2} + \al+} 
 \Big\|\frac{f(n_1, \tau_1)}{\jb{\tau_1 - n_1^3}^\alpha}   h(n, \tau) \Big\|_{ L^{p'}_{  \tau}} 
\end{align*}

\noi
for fixed $n$, $n_2$, and $\tau_2$ with $\ld = \frac{\tau_2 - n_2^3}{3n_2}$.
Now by \eqref{upperbd2} and H\"older inequality in $n$ and then in $\tau_1$, we have
\begin{align*}
\text{LHS of \eqref{DUAL1} } 
& \lesssim \sum_{l=0}^\infty (2^l)^{0-}
\wt{M}_1 
\bigg\| \jb{n_2}^{-1 + \al - s +}
\Big\|\frac{f(n_1, \tau_1)}{\jb{\tau_1 - n_1^3}^\alpha}   h(n, \tau) \Big\|_{ L^{p'}_{ \tau_2, \tau}} 
\bigg\|_{L^{p'}_{|n_2| \sim 2^l} L^{p'}_n}
\\
& \lesssim \sup_{l} 
\wt{M}_1 
\big\| \jb{n_2}^{-1 + \al - s +}
\|\jb{\tau_1 - n_1^3}^{-\alpha} \|_{L^\frac{p}{p-2}_{\tau_1}} 
\|f(n_1, \cdot)\|_{L^p_{\tau_1}}
\|h(n, \cdot) \|_{ L^{p'}_{  \tau}} 
\big\|_{L^{p'}_{|n_2| \sim 2^l} L^{p'}_n},
\end{align*}

\noi
where
$\wt{M}_1 = \sup_{n_2, \tau_2} \|  \jb{\ld + n(n_2-n)}^{-\frac{1}{2} + \al+}  \|_{L^p_n}
< \infty$ in view of Lemma \ref{LEM:Psum} since $2\cdot (\frac{1}{2} -\al -) p > 1$.
We also have 
$\|\jb{\tau_1 - n_1^3}^{-\alpha} \|_{L^\frac{p}{p-2}_{\tau_1}} <C <\infty$,
independently of $n_1$ as before.

Note that if $|n| \sim 2^j$ and $|n_2| \sim 2^l$, 
then we have $|n_1| \sim 2^k$ where $|k - j| \leq 5$ or $|k - l| \leq 5$
since $n = n_1 + n_2$ and $|n_1| \gtrsim |n_2|$.
Then, by H\"older inequality in $n_2$, we have
\begin{align*}
\text{LHS of \eqref{DUAL1} } 
& \lesssim \wt{M}_2 \sup_{l} 
 \bigg( \sum_{j = 0}^\infty \sum_{|n|\sim 2^j}
\|f(n_1, \cdot)\|^{p'}_{L^p_{|n_1|\sim 2^k}L^p_{\tau_1}}
\|h(n, \cdot) \|^{p'}_{ L^{p'}_{  \tau}} 
\bigg)^\frac{1}{p'}\\
&  \lesssim 
 \wt{M}_2 \|f\|_{b^0_{p, \infty}L^p_\tau}\|h\|_{b^0_{p', 1}L^{p'}_\tau}, 
\end{align*}

\noi
where
$\wt{M}_2
=\|\jb{n_2}^{ -1 + \al - s +} \|_{L^\frac{p}{p-2}_{n_2}} < \infty$
since $ (1 - \al + s-)\frac{p}{p-2} > 1$. 

\medskip 

\noi 
$\bullet$ {\bf Case (3):} $\MAX = \jb{\tau_2 - n_2^3}$ and $|n_1| \ll |n_2|$. 
 $ \LRA|n_1| \ll |n_2| \sim |n|$.

In this case, we have
\begin{equation} \label{upperbd3}
\frac{|n|\jb{n}^s}{\jb{n_1}^s\jb{n_2}^s} \frac{1}{\MAX^\frac{1}{2}}
\lesssim \frac{1}{\jb{n_1}^{\frac{1}{2} + s}}.
\end{equation}

\noi

\noi
$\circ$ Subcase (3.a): On $\wt{B}^c_{n_2, \tau_2}$.

If $\jb{\tau_2 - n_2^3 - 3n n_1 n_2} \gtrsim \jb{n_2}^\frac{1}{100}$,
then we have 
$\jb{\tau_2 - n_2^3 - 3n n_1 n_2} \gg \jb{n_1}^\frac{1}{100}$.
By repeating the computation in Subcase (2.a),
 we now have a small negative power of $\jb{n_1} = \jb{n - n_2}$ in \eqref{case2a},
instead of $\jb{ n_2}$, which is still summable in $L^\frac{p}{p-2}_{n_2}$ for each fixed $n$.
Note that if $|n_2| \sim 2^l$, then we have $|n_1| \sim 2^k$  and $|n| \sim 2^j $ 
where $ k = 0, \cdots, l$ and $|j - l | \leq 5$.
Then, it suffice to see 
\begin{align} \label{modification1}
\sum_{l = 0}^\infty \| \jb{n_1}^{0-} F(n, n_2)\|_{L^{p'}_{|n_2| \sim 2^l} L^{p'}_{n}}
& \lesssim \sum_{j = 0}^\infty \sum_{|j - l| \leq 5} \sum_{k = 0}^l (2^k)^{0-}  
\|F(n, n-n_1)\|_{L^{p'}_{|n| \sim 2^j} L^{p'}_{|n_1| \sim 2^k} } \notag \\
& \lesssim \sum_{j = 0}^\infty 
\sup_k    \|F(n, n-n_1)\|_{L^{p'}_{|n| \sim 2^j} L^{p'}_{|n_1| \sim 2^k} }.
\end{align}

Now, suppose $\jb{\tau_2 - n_2^3 - 3n n_1 n_2} \ll \jb{n_2}^\frac{1}{100}$.
Then, we have $w(n_2, \tau_2) \sim \jb{n_1}^\dl$ since $|n_1| \ll  |n|$.
This extra gain of $\jb{n_1}^\dl$ in the denominator of \eqref{DUAL1}
lets us proceed as before.

\medskip

\noi
$\circ$ Subcase (3.b): On $\wt{B}_{n_2, \tau_2}$.
In this case, we have \eqref{away2}
and we can basically proceed as in Subcase (2.b)
 with \eqref{upperbd3} in place of \eqref{upperbd2}.
Using \eqref{modification1}, the modification is straightforward and we omit the details.

\medskip

\noi
$\bullet$ {\bf PART 2:}
Next, we consider the $\ft{Y}^{0, -1}_p$ part 
of the $\ft{W}^{0, -\frac{1}{2}}_p$ norm on the left hand side of \eqref{Wbilinear}.
Define the bilinear operator
$\mathcal{B}_{\theta, b}(\cdot, \cdot) $
by
\begin{align*}
\mathcal{B}_{\theta, b}(f, g)(n, \tau) 
= \frac{1}{2\pi}\sum_{n= n_1 + n_2} 
\frac{1}{\jb{\tau-n^3}^\theta} \intt_{\tau = \tau_1+\tau_2}
\frac{|n|\jb{n}^s}{\jb{n_1}^s\jb{n_2}^s}
\frac{f(n_1, \tau_1) g(n_2, \tau_2) d \tau_1}{ \prod_{j = 1}^2 w(n_j, \tau_j) \jb{\tau_j -n_j^3}^b}.
\end{align*}

If $\MAX = \jb{\tau_1-n_1^3}$ or $\jb{\tau_2-n_2^3}$, 
then by H\"older inequality, we have
\begin{align*}
\text{LHS of \eqref{Wbilinear} }
& = \sup_j \|\mathcal{B}_{-1, \frac{1}{2}}(f, g)(n, \tau) \|_{L^p_{|n|\sim 2^j} L^1_\tau}\\
& \leq
\sup_j \big\|\| \jb{\tau-n^3}^{-\frac{1}{2}-\eps}\|_{L^{p'}_\tau}
\|\mathcal{B}_{-\frac{1}{2}+\eps, \frac{1}{2}}(f, g)(n, \tau)\|_{L^p_\tau} \big\|_{L^p_{|n|\sim 2^j} }\\
& \lesssim
\sup_j \| 
\mathcal{B}_{-\frac{1}{2}+\eps, \frac{1}{2}}(f, g)(n, \tau)\|_{L^p_{|n|\sim 2^j}L^p_\tau },
\end{align*}

\noi 
where we choose $\eps>0$ such that $(\frac{1}{2} +\eps )p' >1$.
For $p = 2+$, we can take $\eps = 0+$.
Then, the proof reduces to Cases (2) and (3), 
where $\jb{\tau-n^3}^\frac{1}{2}$
is replaced by $\jb{\tau-n^3}^{\frac{1}{2}-\eps}$.
Note that this does not affect the argument in Cases (2) and (3).

Now, assume  $\MAX = \jb{\tau-n^3}$.
If $\max(\jb{\tau_1-n_1^3}, \jb{\tau_2-n_2^3}) \gtrsim \jb{\tau-n^3}^\frac{1}{100}$, 
then by H\"older inequality, we have
\begin{align*}
\text{LHS of \eqref{Wbilinear} }
& \leq
\sup_j \big\|\| \jb{\tau-n^3}^{-\frac{1}{2}-\eps}\|_{L^{p'}_\tau}
\|\mathcal{B}_{-\frac{1}{2}, \frac{1}{2}-100\eps}(f, g)(n, \tau)\|_{L^p_\tau} \big\|_{L^p_{|n|\sim 2^j} }\\
& \lesssim
\sup_j \| 
\mathcal{B}_{-\frac{1}{2}, \frac{1}{2}-100\eps}(f, g)(n, \tau)\|_{L^p_{|n|\sim 2^j}L^p_\tau }.
\end{align*}

\noi
Then, the proof reduces to Case (1) 
with $\jb{\tau_j-n_j^3}^\frac{1}{2}$ replaced by $\jb{\tau_j-n_j^3}^{\frac{1}{2}-100\eps}$,
which does not affect the argument.

Lastly, 
if $\max(\jb{\tau_1-n_1^3}, \jb{\tau_2-n_2^3}) \ll \jb{\tau-n^3}^\frac{1}{100}$, 
then by H\"older inequality, we have
\begin{align*}
\text{LHS of \eqref{Wbilinear} }
& \leq
\sup_j \big\|\| \jb{\tau-n^3}^{-\frac{1}{2}} \chi_{\Omega(n)}(\tau-n^3)\|_{L^{p'}_\tau}
\|\mathcal{B}_{-\frac{1}{2}, \frac{1}{2}}(f, g)(n, \tau)\|_{L^p_\tau} \big\|_{L^p_{|n|\sim 2^j} }\\
& \lesssim
\sup_j \| 
\mathcal{B}_{-\frac{1}{2}, \frac{1}{2}}(f, g)(n, \tau)\|_{L^p_{|n|\sim 2^j}L^p_\tau },
\end{align*}

\noi
where the second inequality follows from Lemma \ref{LEM:closetocurve} 
since $-\frac{1}{2} p' = -1 + < -\frac{3}{4}$.
Once again, the proof reduces to Case (1). 

\noi
$\bullet$ {\bf PART 3:}
In this last part, we discuss how to gain a small power of $T$ in 
\eqref{Xbilinear1} by assuming that $u$ or $v$ are supported locally in time.
In Part 1 and 2, we indeed showed
\begin{equation} \label{Xbilinear3}
\|  \dx (uv) \|_{W^{s, -\frac{1}{2}}_p} \lesssim 
 \|  \ft{u}\|_{\ft{X}^{s, b}_p}
\| w \ft{v}\|_{\ft{X}^{s, \frac{1}{2}}_p}
 + \|  w \ft{u}\|_{\ft{X}^{s, \frac{1}{2}}_p}
\| \ft{v} \|_{\ft{X}^{s, b}_p}
\end{equation}

\noi
for some $b \in(0,  \frac{1}{2})$
since we needed the full power of $\frac{1}{2}$ from only one of 
$\jb{\tau - n^3}$, $\jb{\tau_1 - n_1^3}$, or $\jb{\tau_2 - n_2^3}$, i.e. from the maximum one,
and the weight $w(n_j, \tau_j)$ was needed only when $\MAX = \jb{\tau_j - n_j^3}$.
Thus, \eqref{Xbilinear1}
follows once we prove 
\begin{equation} \label{CCtime0}
 \|\eta_{_{2T}} u\|_{X^{s, b}_p} 
 \lesssim T^\theta \| u\|_{X^{s, \frac{1}{2} }_p} 
\end{equation}

\noi
for some $\theta > 0$.
By interpolation, we have 
\begin{equation} \label{CCtime1}
 \| u\|_{X^{s, b}_p} 
 \lesssim  \| u\|^\alpha_{X^{s, 0 }_p}\| u\|^{1-\alpha}_{X^{s, \frac{1}{2} }_p}, 
 \end{equation}

\noi
where $\alpha = 1-2b \in(0, 1)$.
Recall $\ft{\eta_{_{2T}}}(\tau) = 2T \ft{\eta}(2T\tau)$.
Hence, we have $\|\ft{\eta_{_{2T}}}\|_{L^q_\tau} \sim T^\frac{q-1}{q} \|\ft{\eta}\|_{L^q_\tau} \sim T^\frac{q-1}{q}$.
i.e.  we can gain a positive power of $T$ as long as $q>1$.
For fixed $n$, by Young and H\"older inequalities, we have
\begin{align*}
\|  \ft{\eta_{_{2T}}} * \ft{u}(n, \cdot) \|_{L^p_\tau}
\leq \|  \ft{\eta_{_{2T}}} \|_{L^{p'}_\tau} \|\ft{u}(n, \cdot) \|_{L^{\frac{p}{2}}_\tau}
\lesssim T^{\frac{p'-1}{p'}} \|\jb{\tau-n^3}^{-\frac{1}{2}} \|_{L^p_\tau}
\|\jb{\tau-n^3}^{\frac{1}{2}}\ft{u}(n, \cdot) \|_{L^p_\tau}
\end{align*}

\noi
Hence, for $p > 2$, we have
\begin{equation} \label{CCtime2}
\| u\|_{X^{s, 0 }_p} \lesssim T^{\frac{1}{p}} \| u\|_{X^{s, \frac{1}{2} }_p}.
\end{equation}

\noi
Then, \eqref{CCtime0} follows from \eqref{CCtime1} and \eqref{CCtime2}. 
This completes the proof.
\end{proof}

\begin{remark} \label{REM:FLP}\rm
A simple modification of the proof of Proposition \ref{PROP:Wbilinear1}
can be used to establish the local well-posedness of \eqref{KDV} in $\mathcal{F} L^{s, p}= \ft{b}^s_{p, p}$
for some $p = 2+$, $s = -\frac{1}{2}+$ with $sp < - 1$ as well.
Such local solutions can be extended globally a.s. on the statistical ensemble
from the discussion in Section 3.
The modification is straightforward and we omit the details.
\end{remark}

\smallskip

\noindent
{\bf Acknowledgements:} 
The author would like to thank Prof. Luc Rey-Bellet for mentioning the work of Gross \cite{GROSS}
and Kuo \cite{KUO}.

\end{document}